\documentclass[12pt]{article}
\usepackage[utf8]{inputenc}
\usepackage{amssymb}
\usepackage{amsmath}
\usepackage{amsthm}
\usepackage{mathtools}
\usepackage{tikz}
\usepackage{import}
\usepackage{hyperref}

\newtheorem{thm}{Theorem}[section]
\newtheorem{cor}[thm]{Corollary}
\newtheorem{lem}[thm]{Lemma}
\newtheorem{prop}[thm]{Proposition}

\theoremstyle{definition}

\theoremstyle{remark}

\DeclareMathOperator{\lk}{Link}
\DeclareMathOperator{\st}{Star}

\newcommand{\cC }{\mathcal C}

\newcommand{\cN }{\mathcal N}

\newcommand{\cP }{\mathcal P}

\newcommand{\bN}{\mathbb{N}}

\newcommand{\bZ}{\mathbb{Z}}

\def\coloneqq{\mathrel{\mathop\mathchar"303A}\mkern-1.2mu=}

\begin{document}

\title{Intersection of parabolic subgroups in even Artin groups of FC-type}
\author{Yago Antol\'{i}n, Islam Foniqi}

\date{\today}

\maketitle

\begin{abstract}
We show that the intersection of parabolic subgroups of an even finitely generated FC-type Artin group is again a parabolic subgroup.
\end{abstract}


\vspace{0.2cm}\vspace{0.2cm}

\noindent{\bf MSC 2020 classification}: 20F36,20F65.


\section{Introduction}
An {\it  Artin graph} $\Gamma$ is a triple $(V, E, m)$ where $V$ is a set whose elements are called {\it vertices},
$E$ is a set of two-element subsets of $V$ and whose elements are called {\it edges} and  $m\colon E\to \{2,3,4,\dots, \}$ is a function called {\it labelling} of the edges.

Given an Artin graph $\Gamma$,   the corresponding {\it Artin group based on $\Gamma$} (also known as {\it Artin--Tits group}) and denoted by $G_{\Gamma}$ is the group with presentation
$$G_\Gamma \coloneqq \langle\, V \mid \mathrm{prod}(u,v,m({u,v}))=\mathrm{prod}(v,u,m({u,v})) \,\, \forall \{u,v\}\in E\,\rangle,$$
where $\mathrm{prod}(u,v,n)$ denotes the prefix  of length $n$ of the infinite alternating word $uvuvuv\dots$. 

Associated to an Artin graph, we can also construct the {\it Coxeter group based on $\Gamma$} which is the group with presentation 
$$C_\Gamma \coloneqq \langle\, V \mid v^2=1\, \forall v\in V, \, \mathrm{prod}(u,v,m_{u,v})=\mathrm{prod}(v,u,m_{u,v}) \,\, \forall \{u,v\}\in E\,\rangle.$$

An Artin graph $\Gamma$ and the corresponding group $G_\Gamma$ are called of {\it spherical type} if the associated Coxeter group $C_\Gamma$ is finite. 

For $S \subseteq V$, we denote by $G_S$ to the subgroup of $G_\Gamma$ generated by the vertices of $S$. 
Subgroups of this form are called {\it standard parabolic subgroups}, and a theorem of Van der Lek \cite{VdL} shows that $G_S\cong G_\Delta$ where $\Delta$ is the  Artin subgraph of $\Gamma$ induced by $S$.
An Artin graph $\Gamma$ and the corresponding group $G_\Gamma$ are called of {\it FC-type} if every standard parabolic subgroup based on a complete subgraph is of spherical type.

A subgroup $K$ of $G_\Gamma$ is called {\it parabolic} if it is a conjugate of a standard parabolic subgroup.
We say that $K$ is of {\it spherical type} if it is conjugated to a standard parabolic subgroup that is of spherical type.
It was proven by Van der Lek in \cite{VdL} that the class of standard parabolic subgroups is closed under intersection and it is conjectured that the same result holds for the class consisting of all parabolic subgroups.

Let $G_\Gamma$ be an Artin group and $P_1, P_2$ two parabolic subgroups in $G_\Gamma$.
In any of the following cases, $P_1\cap P_2$ is known to be again parabolic:
\begin{enumerate}
\item if $G_\Gamma$ is of spherical type (see \cite{cumplido2020parabolic}),
\item if $G_\Gamma$ is of FC-type and $P_1$ is of spherical type (see \cite{MPV} which generalizes \cite{Morris} where the result was obtained when both $P_1$ and $P_2$ are of spherical type),
\item if $G_\Gamma$ is of {\it large type}, that is $m(\{u,v\})\geq 3$ for all $\{u,v\}\in E$ (see \cite{CMV}),
\item if $G_\Gamma$ is a {\it right-angled Artin group}, that is $m(E) \subseteq \{2\}$ (see \cite{DKR} and \cite{AM} for a generalization to graph products),
\item if $G_\Gamma$ is a {\it (2,2)-free two-dimensional Artin group}, i.e.  $\Gamma$  does not have two consecutive edges labeled by 2 and the geometric dimension of $G_\Gamma$ is two (\cite{Blufstein}),
\item if $G_\Gamma$ is Euclidean of type  $\tilde{A}_n$ or  $\tilde{C}_n$ (\cite{Haettel}).
\end{enumerate} 

We say that an Artin graph $\Gamma = (V,E,m)$ is {\it even} if $m(E)\subseteq 2\bN$. 
The main theorem of this article is:
\begin{thm}\label{thm: intersections}
Let $\Gamma=(V,E,m)$ be an even, finite Artin graph of FC-type.  
The intersection of two parabolic subgroups of $G_\Gamma$ is parabolic.
\end{thm}
It is an standard argument to deduce from this theorem that the intersection of arbitrary many parabolic sugroups is again parabolic (see Corollary \ref{cor: arbitrary intersection}).

The class of even FC-type Artin groups includes the class of right-angled Artin groups (raags for short), and  they possess some similar properties.
On one side, we understand well the case when $\Gamma$ is a complete even FC-type Artin graph. 
This implies that $G_\Gamma$ is a direct product of $(\leq 2)$-generated Artin groups (in the case of raags $G_\Gamma$ is free abelian).
On another side, every parabolic subgroup $P$ of an even (FC-type) Artin group $G_\Gamma$ is a {\it retract} i.e. there is a homomorphism $\rho\colon G \to P$ such that $\rho$ restricted to $P$ is the identity.
 
With these two properties, one can decompose even FC-type Artin groups into direct product and amalgamated free products, and in the latter case we use the geometry of the Bass-Serre tree to deduce properties of the intersections of parabolic subgroups.
In fact, we use these two facts in Section \ref{sec: retractions} to reduce the proof of Theorem \ref{thm: intersections} to the case where the parabolic subgroups are conjugate to the same standard parabolic $G_A$  of $G_\Gamma$ and moreover the graph $\Gamma$ satisfies $\st(x)=V$, for all $x\in V-A$.
In this setting, we deduce that the intersection of parabolic subgroups of raags is parabolic and we note that this proof is different from the ones of \cite{DKR} and \cite{AM} which use normal forms.

We remark that in \cite{MPV} the action on the Bass-Serre tree is used in a similar spirit as here.

However, for proving  Theorem  \ref{thm: intersections} we use more properties of even FC-type Artin groups. 
We show that under some circumstances, the kernel of the retractions of standard parabolic subgroups are again even FC-type Artin groups. 
Let $\Gamma$ be an even FC-type Artin graph. 
In \cite{BMP} it was shown that for every $v\in  V$, the canonical retraction $\rho: G_\Gamma\to G_{V-\{v\}}$, has free kernel and they give a description of a free basis. 
With this result they deduce that $G_\Gamma$ is poly-free.
We use these kernels and also the kernels of the retractions $\rho\colon G_\Gamma\to G_v$, which as we will show are again even FC-type Artin groups under certain conditions on $\lk(v)$. These are the main results of Section \ref{sec: kernels}, where we provide precise description of these kernels.

We prove Theorem \ref{thm: intersections} in Section \ref{sec: intersection}. 
We remark that in contrast with \cite{cumplido2020parabolic, Morris} our proof does not make use of Garside theory. 
The paper is almost self-contained, we rely on the Bass-Serre theorem, the Redemiester-Schreier method and the description of kernels of \cite{BMP}.

We will begin setting some notation.

\section{Notation}
\label{sec: notation}
Let $\Gamma=(V,E,m)$ be an Artin graph.
Note that $V,E$ are the vertices and edges respectively of a simplicial graph. 
We will use standard terminology of graphs:
 for $v\in V$, the set $\lk(v)=\{u: \{v,u\}\in E\}$ is called the {\it link } of $v$. 
The set $\st(v)=\lk(v)\cup \{v\}$ is called the {\it star} of $v$.
Given a subset $S$ of $V$ the {\it subgraph induced by} $S$, and denoted $\Gamma_S$, is the Artin graph with vertices $S$, edges $E'=\{\{u,v\}\in E\mid u,v\in S\}$ and labelling that consists on restricting $m$ to $E'$.

We note that the notion of being a (standard) parabolic subgroup of $G_\Gamma$ depends on the presentation defined by $\Gamma$ and not on the isomorphism class of $G_\Gamma$, so if needed, we will say that a subgroup is {\it $\Gamma$-parabolic}.
This terminology will be relevant in the proof of our main theorem, as we will use that some parabolic subgroups of $G_\Gamma$ are also parabolic  in $G_\Delta$, where $G_\Delta$ is an Artin subgroup of $G_\Gamma$.

For an edge $\{u,v\}\in E$ we denote $m(\{u,v\})$ by $m_{u,v}$ to simplify the notation (note that $m_{u,v}= m_{v,u}$).

Suppose that $\Gamma$ is even, i.e.  $m(E)\subseteq 2\bN$, then for any $S \subseteq V$, one has a retraction $$\rho_S:G_{\Gamma}\longrightarrow G_S$$ defined on the generators of $G_{\Gamma}$ as:
$\rho_S(s) = s$ for $s\in S$, and $\rho_S(v) = 1$ for $v\in V\Gamma - S$. 
When $S=\{v\}$, we might write $\rho_{\{v\}}$ as $\rho_v$. 
Moreover, as $\langle v\rangle \cong \bZ$ via $v^n\mapsto n$, in many cases we use $\bZ$ as the co-domain of $\rho_v$ without mentioning it. 
The use of this isomorphism should be clear from the context.

There is a simple condition for having an {\it even Artin graph of type FC:}  $m$ is an even labelling of $E$ and for any triangle with edges $\{u,v\},\{v,w\},\{w,u\}\in E$, at least two of $m_{u,v},m_{v,w}, m_{w,u}$ are equal to two (see \cite[Lemma 3.1]{BMP}).

If $S,T$ are subsets of a group $G$ we write $S^T$ to denote the set $\{tst^{-1}: t\in T, s\in S\}$.
If~$S=\{s\}$ we just write $s^T$ to mean $\{s\}^T$ and similarly, if $T=\{t\}$ we just write $S^t$ instead of $S^{\{t\}}.$

\section{Even labelling and retractions.}
\label{sec: retractions}
Throughout this section $\Gamma=(V,E,m)$ is an  even Artin  graph.
Some of the results of this section have been proved in a more general context, however as the proof in the even case is very elementary, we have chosen to give the proof to make the paper as much self-contained as possible.
For example, the next lemma holds for any Artin group \cite{VdL}.

\begin{lem}\label{lem: intersection_of_standard_parabolic_subgroups}
Let $A, B \subseteq V$. The following equality holds:
$$G_A \cap G_B = G_{A \cap B}.$$
\end{lem}
\begin{proof}
Let $\rho_A$, and $\rho_B$ be the corresponding retractions for $G_A$, and $G_B$ respectively. Consider the compositions $\rho_A \circ \rho_B$ and $\rho_B \circ \rho_A$. 
When applying them to $v \in V$ we notice that~$(\rho_A \circ \rho_B)(v) = \rho_{A \cap B}(v) =  (\rho_B \circ \rho_A)(v)$. Extending to morphisms on the group $G_{\Gamma}$ we obtain a commutative diagram of retractions, in the form: $\rho_A\circ \rho_B = \rho_B\circ \rho_A = \rho_{A \cap B}.$

As $G_{A \cap B} \subseteq G_A$ and $G_{A \cap B} \subseteq G_B$ one has $G_{A \cap B} \subseteq G_A \cap G_B$.

To show the other inclusion $G_A \cap G_B \subseteq G_{A \cap B}$, pick an element~$x \in G_A \cap G_B$. 
One has~$x \in G_A$ and~$x \in G_B$, so~$\rho_A(x) = \rho_B(x) = x$. 
Now using that retractions commute we obtain:
$$\rho_{A \cap B}(x) = (\rho_A\circ \rho_B) (x) = \rho_A(\rho_B (x)) = \rho_A(x) = x.$$
As $\rho_{A \cap B}$ is a retraction, we have $x \in G_{A \cap B}$, as required.
\end{proof}


\begin{lem}\label{lem: proper_inclusions_parabolics}
Let $A, B \subseteq V$ and $g, h \in G$. Then $gG_Ag^{-1} \subsetneq hG_Bh^{-1}$ implies $A\subsetneq B$.
\end{lem}

\begin{proof}
Conjugating by $h^{-1}$, we can write the proper inclusion $gG_Ag^{-1} \subsetneq hG_B h^{-1}$ in the equivalent form $fG_Af^{-1} \subsetneq G_B$, for $f = h^{-1}g$. 
Applying $\rho_B$ we obtain:
$$fG_Af^{-1} = \rho_B(fG_Af^{-1}) =  \rho_B(f)G_{A\cap B}\rho_B(f)^{-1} \subsetneq G_B.$$
So, the proper inclusion $fG_Af^{-1} \subsetneq G_B$ is equivalent to the proper inclusion 
$$\rho_B(f)G_{A\cap B}\rho_B(f)^{-1} \subsetneq G_B,$$ which after conjugating by $\rho_B(f)^{-1}$ becomes equivalent to $G_{A\cap B} \subsetneq G_B$, and this implies that $A \cap B \subsetneq B$.
	
Instead, applying $\rho_A$ to $fG_Af^{-1} \subsetneq G_B$ we obtain
$$G_A = \rho_A(f)G_A\rho_A(f)^{-1} = \rho_A(fG_Af^{-1}) \subseteq \rho_A(G_B) = G_{A\cap B}.$$
The inclusion $G_A \subseteq G_{A\cap B}$ implies $A\subseteq A\cap B$.
Ultimately $A\subseteq A\cap B \subsetneq B$, which means that $A \subsetneq B$, as required.
\end{proof}

In the next lemma, we reduce the problem of showing that an intersection of parabolic subgroups is again parabolic, to deciding whether the intersection of two conjugates of a standard parabolic subgroup $G_A$ is again parabolic. Once again, we will make use of retractions.

\begin{lem}\label{lem: intersection_of_two_parabolic_subgroups}
Let $f,g\in G$ and $A,B\subseteq V$.
There exist $a\in G_A$ and $b \in G_B$ such that $$fG_Af^{-1}\cap g G_Bg^{-1} = faG_Ca^{-1}f^{-1}\cap gb G_Cb^{-1}g^{-1},$$
where $C= A\cap B$. 
\end{lem}
\begin{proof}
One has the equality 
$$fG_Af^{-1} \cap gG_B g^{-1} = f[G_A \cap (f^{-1}g) G_B (f^{-1}g)^{-1}]f^{-1}.$$
Set $h = f^{-1}g$ and consider $P = G_A \cap hG_B h^{-1}$.
Using $P\subseteq G_A$, and $G_A \cap G_B = G_{A \cap B}$ (see Lemma \ref{lem: intersection_of_standard_parabolic_subgroups}) we obtain:
\begin{align*}
P = \rho_A(P) = \rho_A(G_A \cap hG_B h^{-1}) & \subseteq \rho_A(G_A)\cap \rho_A(hG_B h^{-1}) \\
& = G_A \cap \rho_A(h)\rho_A(G_B)\rho_A(h^{-1}) \\
& = \rho_A(h)G_{A\cap B}\rho_A(h)^{-1}.
\end{align*}
Setting $a= \rho_A(h)\in G_A$ and $A\cap B = C$ we can write the inclusion above as~$P\subseteq aG_{C}a^{-1}$, and we notice that $aG_{C}a^{-1} \subseteq G_A$. Also,~$P = G_A \cap hG_B h^{-1}$, so we have
\begin{align*}
P = (G_A \cap hG_B h^{-1}) \cap aG_{C}a^{-1} & = h G_B h^{-1} \cap (G_A \cap aG_{C}a^{-1}) \\
& = h G_B h^{-1} \cap aG_{C}a^{-1}.
\end{align*}
Multiplying the last equation by $h^{-1}$ and denoting $P' = h^{-1}Ph$, $k = h^{-1}a$, we obtain:
$P' = G_{B} \cap kG_{C} k^{-1}.$

Applying the same procedure as for $P$ above, we obtain:
\begin{align*}
    P' = \rho_B(P') & = \rho_B(G_B \cap kG_{C} k^{-1}) \\
    & \subseteq \rho_B(G_B)\cap \rho_B(kG_{C} k^{-1}) \\
    & = \rho_B(k)G_{B\cap C}\rho_B(k)^{-1}\\
    & = \rho_B(k)G_{C}\rho_B(k)^{-1}.
\end{align*}
Setting $b = \rho_B(k)\in G_B$ we express the inclusion above as $P' \subseteq bG_{C} b^{-1} \subseteq G_B$.
Putting together~$P' = G_{B} \cap kG_{C} k^{-1}$ and~$P' \subseteq bG_{C} b^{-1}$ we have:
$$P' = (G_{B} \cap kG_{C} k^{-1}) \cap bG_{C} b^{-1} =  h^{-1}aG_{c} a^{-1}h \cap (G_{B} \cap bG_{C} b^{-1}).$$
Using $G_{B} \cap bG_{C} b^{-1} = bG_{C} b^{-1}$, and $P' = h^{-1}Ph$ we ultimately have:
$$P = aG_{C} a^{-1} \cap hbG_{C} b^{-1}h^{-1}.$$

Turning back, we have $fG_Af^{-1}\cap g G_Bg^{-1} = fPf^{-1}$, and $h=f^{-1}g$, so we obtain: 
$$fG_Af^{-1}\cap g G_Bg^{-1} = faG_Ca^{-1}f^{-1}\cap gb G_Cb^{-1}g^{-1},$$
where $C= A\cap B$, as desired.
\end{proof}

The next lemma holds for any Artin group, see \cite[Proposition 2.6]{MPV}.
The proof in the even case is much simpler.

\begin{lem}\label{lem: containing implies equality}
Let  $g,h\in G_{\Gamma}$ and $A\subseteq V$.
If $gG_Ag^{-1}\leqslant hG_Ah^{-1}$ then $gG_Ag^{-1}=  hG_Ah^{-1}$.
\end{lem}
\begin{proof}
We have that  $gG_Ag^{-1}\leqslant hG_Ah^{-1}$ if and only if $h^{-1}gG_Ag^{-1}h\leqslant G_A$. In particular,
$h^{-1}gG_Ag^{-1}h= \rho_A(h^{-1}gG_Ag^{-1}h) =G_A$. The lemma follows.
\end{proof}

\begin{cor}
Let $A,B\subseteq V$ and $f,g\in G_\Gamma$.
Let $H=fG_Af^{-1}$ and $K=gG_Bg^{-1}$ be parabolic subgroups of $G_\Gamma$. If $H=K$ then $A=B$.
\end{cor}
In particular, if $K$ is a parabolic subgroup of $G_\Gamma$, there is a unique $S\subseteq V$ such that $K$ is conjugate to $G_S$. 
In that event, we say that $K$ is {\it parabolic over} $S$. We note that if $K = fG_Sf^{-1}$, where $f\in G_\Gamma$, then $K$ is also a retract of $G_\Gamma$, with the retraction homomorphism $$\rho_K =\rho_S^f \colon G_\Gamma \longrightarrow K = fG_Sf^{-1},\qquad \rho_S^f(g) \coloneqq f\rho_S(f^{-1}gf)f^{-1}$$ for all $g \in G_\Gamma$.
We will preferably use the notation $\rho_K$, however we might use $\rho_S^f$ if we want to emphasize the choice of the element in $fN_{G_{\Gamma}}(G_S)$,   the coset of the normalizer of $G_S$,  that we are using to conjugate.

\begin{lem}\label{lem: link exterior}
Let $A\subseteq V$ and $g\in G_\Gamma$.
 Suppose that $G_A \cup gG_Ag^{-1}$
is not contained in a proper parabolic
subgroup of $G$ and for some $x\in V- A$, one has that  $A$ is not contained in $\lk(x)$.
Then $G_A \cap  gG_Ag^{-1}$
is  contained in a parabolic subgroup over a
proper subset of $A$. 
\end{lem}
\begin{proof}
Let $A\subseteq V$ and $g\in G_\Gamma$ as in the hypothesis.
Consider $P = G_A \cap gG_Ag^{-1}$. 
Assume that there is an $x\in V - A$ with the property $\lk(x)\not\supseteq A$. 

If $G_{V- \{x\}} = gG_{V- \{x\}}$, then $g \in G_{V- \{x\}}$. This means that both $G_A$ and $gG_Ag^{-1}$ are parabolic subgroups in $G_{V- \{x\}}$, and hence $G_A \cup gG_Ag^{-1}$ is contained in the proper parabolic subgroup $G_{V- \{x\}}$ of $G$. 
This contradicts the assumptions of the proposition, so suppose that $G_{V- \{x\}} \neq gG_{V- \{x\}}$.
From the presentation, one has an splitting of $G_\Gamma$ as an amalgamated free product:
$$G_\Gamma = G_{\st(x)}*_{G_{\lk(x)}} G_{V- \{x\}}.$$
Consider the Bass-Serre tree $T$ corresponding to this splitting (see for example \cite{DicksDunwoody}). 
There are two types of vertices in $T$: left cosets of $G_{\st(x)}$, and left cosets of $G_{V- \{x\}}$ in $G$. Only vertices of different type can be adjacent in $T$.
The group $G$ acts naturally on $T$, without edge inversions. Moreover, the vertex stabilizers correspond to conjugates of $G_{\st(x)}$ and conjugates of $G_{V- \{x\}}$ for the respective type of vertices, while the edge stabilizers are conjugates of ${G_{\lk(x)}}$. 

In the tree $T$, both $G_{V- \{x\}}$ and $gG_{V- \{x\}}$, are distinct vertices of the same type. Their stabilizers are $G_{V- \{x\}}$, and $gG_{V- \{x\}}g^{-1}$ respectively. As we are on a tree, there is a unique geodesic $p$ in $T$ connecting $G_{V- \{x\}}$ and $gG_{V- \{x\}}$. 

Since $x\not\in A$, we have:
$G_A \subseteq G_{V- \{x\}}$ and $gG_Ag^{-1} \subseteq gG_{V- \{x\}}g^{-1}$,
which means that our parabolic subgroups~$G_A$ and $gG_Ag^{-1}$  stabilize the vertices corresponding to the cosets $G_{V- \{x\}}$ and $gG_{V- \{x\}}$ respectively. The intersection $G_A \cap gG_Ag^{-1}$ stabilizes the geodesic $p$ connecting those vertices, and hence it stabilizes any edge belonging to $p$.
Since stabilizers of edges in $T$ are conjugates of ${G_{\lk(x)}}$, we have:
$$P = G_A \cap gG_Ag^{-1} \subseteq hG_{\lk(x)}h^{-1}$$
for some $h\in G$. 
Now one can write $P$ as:
\begin{equation*}
P= G_A \cap gG_Ag^{-1} \cap hG_{\lk(x)}h^{-1} = (G_A \cap hG_{\lk(x)}h^{-1}) \cap (gG_Ag^{-1} \cap hG_{\lk(x)}h^{-1})
\end{equation*}

By Lemma \ref{lem: intersection_of_two_parabolic_subgroups}, one can express $G_A \cap hG_{\lk(x)}h^{-1}$ as an intersection of two parabolic subgroups over $\lk(x)\cap A \subsetneq A$ (because $\lk(x)\not\supseteq A$), hence $P$ is contained in a parabolic subgroup over a proper subset of $A$. This completes the proof.
\end{proof}

\begin{lem}\label{lem: reducing subgraphs}
Let $\Delta$ be a subgraph of $\Gamma$, $A\subseteq V_\Delta$ and $g,t\in G_\Gamma$. 
If $G_A\cup gG_Ag^{-1}$ is contained in $tG_\Delta t^{-1}$,
then there is $h\in G_\Delta$ such that:
\begin{enumerate}
\item[(i)] $G_A=hG_Ah^{-1}$ if and only if $G_A=gG_Ag^{-1}$.
\item[(ii)] $G_A\cap hG_Ah^{-1}$ is $\Delta$-parabolic if and only if $G_A\cap gG_Ag^{-1}$ is $\Gamma$-parabolic.
\item[(iii)] $G_A\cap hG_Ah^{-1}$ is contained on a $\Delta$-parabolic over a proper subset of $A$  if and only if $G_A\cap gG_Ag^{-1}$ is $\Gamma$-parabolic over a proper subset of $A$.
\end{enumerate} 
\end{lem}
\begin{proof}
Suppose that the inclusion $G_A\cup gG_Ag^{-1} \subseteq tG_\Delta t^{-1}$ holds. Multiplying by $t^{-1}$ we obtain $t^{-1}G_A t \cup t^{-1} gG_A g^{-1}t\subseteq G_\Delta$. Applying $\rho_\Delta$, and recalling that $A\subseteq V_\Delta$, we get: 
$$t^{-1}G_A t =  \rho_\Delta(t^{-1}) G_{A} \rho_\Delta(t)\text{  and  }t^{-1} gG_A g^{-1}t = \rho_{\Delta}( t^{-1} g)G_A \rho_{\Delta}(g^{-1}t).$$
Let $f_1=\rho_\Delta(t^{-1})$ and $f_2=\rho_\Delta (t^{-1} g)\in G_\Delta$. 
We have that 
$$t^{-1}G_A t \cap t^{-1} gG_A g^{-1}t =  f_1G_Af_1^{-1} \cap f_2G_A f_2^{-1}.$$
Observe that  $t^{-1}G_A t \cap t^{-1} gG_A g^{-1}t$ is $\Delta$-parabolic (resp. contained in a parabolic subgroup over a subset of $A$) if and only if $G_A\cap gG_Ag^{-1}$ is $\Delta$-parabolic (resp. contained in a parabolic subgroup over a subset of $A$).

We will take  $h=f_1^{-1}f_2\in G_\Delta$. Observe that  $f_1G_Af_1^{-1} \cap f_2G_A f_2^{-1}$ is $\Delta$-parabolic (resp. contained in a parabolic subgroup over a subset of $A$) if and only if $G_A\cap hG_Ah^{-1}$ is $\Delta$-parabolic (resp. contained in a parabolic subgroup over a subset of $A$).

We prove (i). Note that $G_A=gG_Ag^{-1}\Leftrightarrow t^{-1}G_A t = t^{-1}gG_Ag^{-1}t \stackrel{(*)}{\Leftrightarrow} \rho_{\Delta} (t^{-1}G_A t) = \rho_\Delta ( t^{-1}gG_Ag^{-1}t ) \Leftrightarrow f_1 G_A f_1^{-1} = f_2 G_A f_2^{-1} \Leftrightarrow G_A = hG_Ah^{-1}$. The equivalence $(*)$ uses that $t^{-1}G_At$ and $t^{-1}gG_A g^{-1}t$ are contained in $G_\Delta$.

To show (ii), by the previous discussion, it is enough to show that $t^{-1}G_A t \cap t^{-1} gG_A g^{-1}t$ is $\Gamma$-parabolic if and only if $f_1G_Af_1^{-1} \cap f_2G_A f_2^{-1}$ is $\Delta$-parabolic.

As $\Delta$  is a subgraph of $\Gamma$, any  $\Delta$-parabolic subgroup of $G_\Delta$  is a $\Gamma$-parabolic subgroup of $G_\Gamma$. Thus, if $f_1 G_Af_1^{-1}\cap f_2 G_A f_2^{-1}$ is $\Delta$-parabolic, then it is also $\Gamma$ parabolic.
Conversely, if  there  for some $B\subseteq G_\Gamma$ and some $d\in G_\Gamma$ such that,  $f_1G_Af_1^{-1}\cap f_2G_Agf_2^{-1}= d G_Bd^{-1}$. 
As $ f_1G_Af_1^{-1} \cap f_2G_A f_2^{-1}\subseteq G_\Delta$, applying $\rho_\Delta$,
and noting that $B\subseteq A \subseteq \Delta$
we get that $f_1G_Af_1^{-1}\cap f_2G_Af_2^{-1}=f_3G_{B}f_3^{-1}$ where 
$f_3=\rho_\Delta(d)$.

A similar argument as above shows (iii).
\end{proof}

Suppose that $\cC$ is a class of even Artin graphs closed by taking subgraphs, satisfying the following
\begin{equation}
\label{eq: base of induction}
\begin{multlined}
\text{ for all $\Gamma\in\cC, \, A\subseteq V,\, g\in G_\Gamma$ such that for all $x\in V-A, \, \st(x)=V$ one has that}\\ 
\text{  $G_A=gG_Ag^{-1}$ or $G_A\cap gG_Ag^{-1}\leqslant dG_Bd^{-1}$ for some $B\subsetneq A$ and $d\in G_\Gamma$.}
\end{multlined}
\end{equation}

\begin{prop}\label{prop: reduction to stars}
If $\cC$ is a class of even Artin graphs closed by taking subgraphs, and satisfying \eqref{eq: base of induction}, then for every $\Gamma\in \cC$ the intersection of two $\Gamma$-parabolic subgroups of $G_\Gamma$ is parabolic. 
\end{prop}
\begin{proof}
Let $\Gamma\in \cC$ and let $P,Q$ be two parabolic subgroups of $G_\Gamma$. 
By Lemma \ref{lem: intersection_of_two_parabolic_subgroups} we can suppose that there is $A\subseteq V$ and $h_1,h_2\in G_\Gamma$ such that $P\cap Q = h_1 G_Ah_1^{-1} \cap h_2 G_A h_2^{-1}$.
Therefore $P\cap Q$ is parabolic if and only if $G_A\cap gG_Ag^{-1}$ is parabolic, where $g=h_1^{-1}h_2$.

If $G_A\cup gG_Ag^{-1}$ is contained in a proper parabolic subgroup of $G_\Gamma$,  then by Lemma \ref{lem: reducing subgraphs} we can replace $\Gamma$ by a proper subgraph $\Delta$ and replace $g$ by some $h\in G_\Delta$.
Note that $\Delta$ is still in the class $\cC$.

Therefore, we can return to the initial notation  and further assume that hat $G_A\cup gG_Ag^{-1}$ is not contained over a proper parabolic of $G_\Gamma$.

We will show that for every $A\subseteq V$ finite and any $g\in G_\Gamma$, $G_A\cap g G_Ag^{-1}$ is parabolic.
Our proof is by induction on $|A|$. 
If $|A|=0$, then $G_A$ is trivial and the result follows.
We now assume that $|A|>0$ and that for parabolic subgroups over smaller sets the result holds.
We remark that the induction hypothesis is equivalent to saying that for any $B\subseteq V$, $|B|<|A|$ and any $g_1,g_2\in G_\Gamma$, $g_1G_Bg_1^{-1} \cap g_2G_Bg_2^{-1}$ is parabolic. 

If there is $x\in V-A$ such that $A$ is not contained in $\lk(x)$, then Lemma \ref{lem: link exterior} implies that $G_A\cap gG_Ag^{-1}\leqslant dG_Bd^{-1}$ for some $B\subsetneq A$ and some $d\in G_\Gamma$.
Therefore, by Lemma \ref{lem: intersection_of_two_parabolic_subgroups} there are $a,a'\in G_A$ and $b,b'\in G_B$ such that 
\begin{align*}
G_A\cap gG_Ag^{-1} &= G_A\cap gG_Ag^{-1} \cap d G_B d^{-1}\\
& = (G_A \cap d G_B d^{-1}) \cap (gG_Ag^{-1} \cap d G_B d^{-1})\\
& = (aG_{A\cap B}a^{-1} \cap d bG_{A\cap B} b^{-1} d^{-1})\cap (ga'G_{A\cap B} {a'}^{-1}g^{-1} \cap db'G_{A\cap B}(b')^{-1} d^{-1} )\\
& = (aG_{B}a^{-1} \cap d G_{B} d^{-1})\cap (ga'G_{B} {a'}^{-1}g^{-1} \cap dG_{ B} d^{-1} ).
\end{align*}
As $|B|<|A|$, by induction, $aG_{B}a^{-1} \cap d G_{B} d^{-1}$ and $aG_{B}a^{-1} \cap d G_{B} d^{-1}$ are parabolic subgroups of $G_\Gamma$ over subsets of $B$. 
Say that $aG_{B}a^{-1} \cap d G_{B} d^{-1} = g_1 G_{B_1} g_1^{-1}$ and $aG_{B}a^{-1} \cap d G_{B} d^{-1} = g_2 G_{B_2} g_2^{-1}$.
Thus, using again Lemma \ref{lem: intersection_of_two_parabolic_subgroups}, we get that $$G_A\cap gG_Ag^{-1}  =  g_1 G_{B_1} g_1^{-1} \cap  g_2 G_{B_2} g_2^{-1} = g_1xG_C x^{-1}g_1^{-1}\cap g_2yG_Cy^{-1}g_2^{-1}$$
where $x\in G_{B_1}$, $y\in G_{B_2}$ and $C=B_1\cap B_2$. As $|C|<|A|$, using again induction, we get that $G_A\cap gG_Ag^{-1}$ is parabolic.

So, we assume that for all $x\in V-A$, $A\subseteq \lk(x)$.
We now argue by induction on $N=\sharp \{x\in V-A : \st(x)\neq V\}$.
In the case $N=0$, as we are in the class $\cC$ that satisfies \eqref{eq: base of induction}, we have that either $G_A=gG_Ag^{-1}$ and  hence $G_A\cap gG_Ag^{-1}$ is parabolic, or  $G_A\cap gG_Ag^{-1}\leqslant d G_B d^{-1}$ for some $B\subsetneq A$. 
As before, the latter implies that $G_A\cap gG_Ag^{-1}$ is an intersection of four parabolics over $B$, with $|B|<|A|$ and arguing as above, we get that $G_A\cap gG_Ag^{-1}$ is parabolic.
So assume that $N>0$ and the result is known for smaller values of $N$ and $A$.

Let $x,y\in V-A$ not linked by an edge.  Let $X=\st(x)$, $Y=V-\{x\}$ and $Z=\lk(x)$.
Then we have an amalgamated free product:
$$G = G_X*_{G_Z} G_{Y}$$
and consider the associated Bass-Serre tree $T$ corresponding to this splitting.

Consider the edges $G_Z$ and $gG_Z$ on $T$. 
Let $G_Z, g_1G_Z, \ldots, g_nG_Z, gG_Z$ be sequence of edges in the unique geodesic in $T$ connecting $G_Z$ and $gG_Z$.
If $G_Z = gG_Z$ (i.e $n=0$) and taking into account that $A\subseteq Z$, we have that  $G_A\cup g G_Ag^{-1}$ is contained in $G_Z$, which is a proper parabolic of $G_\Gamma$ and this contradicts our hypothesis.
So we assume that $n\geq 1$. 
By the construction of $T$, one has either $g_i^{-1}g_{i+1} \in G_{X}$ or $g_i^{-1}g_{i+1} \in G_{Y}$, for any~$i = 0, \ldots, n$ where $g_0 = 1$ and $g_{n+1} = g$. 
The intersection $G_A\cap gG_A g^{-1}$ stabilizes the endpoints of the geodesic path, hence it stabilizes the whole path.
As the stabilizer of a geodesic is the intersection of stabilizers of its edges, we have the equality
$$G_A\cap gG_A g^{-1} = G_A\cap g_1G_Z g_1^{-1} \cap \ldots \cap g_nG_Zg_n^{-1}\cap gG_A g^{-1}.$$
By Lemma \ref{lem: intersection_of_two_parabolic_subgroups}, (applied to $G_A\cap g_i G_Zg_i^{-1})$, and the fact that $A\subseteq Z$ we have that there are $z_i\in G_Z$ such that $G_A\cap g_iG_Zg_i^{-1}$ are equal to $G_A\cap g_iz_i G_A z_i^{-1}g_i^{-1}$. 

Note that $(g_iz_i)^{-1}(g_{i+1}z_{i+1}) = z_i^{-1}(g_i^{-1}g_{i+1})z_{i+1}$, so replacing $g_iz_i$ by $g_i$ we still have that $g_i^{-1}g_{i+1} \in G_{X}$ or $g_i^{-1}g_{i+1} \in G_{Y}$, for any~$i = 0, \ldots, n$ where $g_0 = 1$ and $g_{n+1} = g$. 
Hence:
$$G_A\cap gG_A g^{-1} = G_A\cap g_1G_Ag_1^{-1} \cap \ldots \cap g_nG_Ag_n^{-1}\cap gG_A g^{-1}.$$

The intersections $g_iG_Ag_i^{-1} \cap g_{i+1}G_Ag_{i+1}^{-1}$ can be expressed as:
$$g_iG_Ag_i^{-1} \cap g_{i+1}G_Ag_{i+1}^{-1} = g_i[G_A \cap g_i' G_Ag_i'^{-1}]g_i^{-1},$$
where $g_i' = g_i^{-1}g_{i+1} $ is either in $G_{X}$, or in $G_{Y}$. 
As the number of vertices in $X-A$ (resp. $Y-A$) whose star is not $X$ (resp. $Y$) is less than $N$, we can apply induction and the intersections $G_A \cap g_i' G_Ag_i'^{-1}$ are either equal to $G_A$ or are contained in a parabolic subgroup over a proper subset of $B$ of $A$.
If we have equality for $i=1,\dots, n$, then $G_A\cap g G_Ag^{-1}=G_A$, in other case, $G_A\cap g G_Ag^{-1}$ is contained in a parabolic subgroup over a proper subset $B$ of $A$ and we conclude by induction on $|A|$.
\end{proof}

\begin{cor}
Let $\Gamma$ be right-angled Artin  graph. 
The intersection of two parabolic subgroups of $G_\Gamma$  is parabolic.
\end{cor}
\begin{proof}
Let $\cC$ be the family of finite right-angled Artin graphs. 
Clearly $\cC$ is closed under subgraphs.
Now take $A\subseteq V,\, g\in G_\Gamma$ such that for all $x\in V-A, \, \st(x)=V$.
Then $G_\Gamma$ is a direct product of $G_A$ and $G_{V-A}$ and thus 
 $G_A=gG_Ag^{-1}$. 
 Then $\cC$ satisfies \eqref{eq: base of induction} and the corollary follows from Proposition \ref{prop: reduction to stars}.
\end{proof}

\section{Even FC-Artin labelling and kernels}
\label{sec: kernels}
Through this section $\Gamma =(V,E,m)$ is an even Artin graph of FC-type.

Let $x\in V$.  In this section we describe the kernels of the retractions $\rho_{\{x\}}$ and $\rho_{V-\{x\}}$.
The kernel of $\rho_{V-\{v\}}$ was described in \cite{BMP}, and turns out to be a free group, we will just recall their result.
Our main contribution in this section is showing that $\ker \rho_{\{x\}}$ is isomorphic to an even FC-type Artin group $G_\Delta$ when $\st(x)=V$. 
The construction of the Artin graph $\Delta$ and the isomorphism will be explicit and will allow us in the next section to show that certain $\Delta$-parabolic subgroups of $G_\Delta$ are also $\Gamma$-parabolic (as subgroups of $G_\Gamma$).

\subsection{Kernel of a retraction onto a vertex }\label{subsec: kernel vertex}
Let $x\in V$  and $\rho\coloneqq \rho_{\{x\}}\colon G_{\Gamma} \to \langle x \rangle$ the associated retraction.
We assume that at least one of the following holds:
\begin{enumerate}
\item[(a)] $\st(x)=V$,
\item[(b)] for all $u\in L=\lk(x)$, $m_{u,x}=2$.
\end{enumerate}

We will see that under one of the previous conditions
\footnote{In fact, with  hypothesis (b), we do not use that $\Gamma$ is of FC-type} $K=\ker \rho_{x}$ is isomorphic to $G_\Delta$, 
where $\Delta= (V_\Delta, E_\Delta, m^{\Delta})$ is an even FC-type Artin graph.
Moreover, $V_\Delta$ will come with an indexing: $i\colon V_\Delta \to \bZ$.
We will say that $P\leqslant G_\Delta$ is {\it index parabolic} (with respect to $i$) 
if there is $n\in \bZ$, $S\subseteq i^{-1}(n)$ and $g\in G_\Delta$ such that $P= gG_Sg^{-1}$.

Let $L=\lk(x)\subseteq V$ and $B=V-(\st(x))$.
For $u\in L$, let $k_u= m_{u,x}/2$.
Let $\Delta$ be the graph with vertex set 
$$V_\Delta = \left(\bigcup_{u\in L}\{u\}\times \{0,1,\dots, k_u-1\}\right) \cup \left(\bigcup_{u\in B}\{u\}\times \bZ\right).$$
We define the indexing $i\colon V_\Delta \to \bZ$ as $i(v,n)=n$.
For simplicity, we write a vertex $(v,n)$ as $v_n$.
For future use, we set the following terminology. A vertex $v_i\in V_\Delta$ will be called {\it of type }$v\in V$ and of {\it index } $i$.

The edge set of $\Delta$ is 
$$E_\Delta = \{\{u_n,v_m\}\} : u_n,v_m\in V_\Delta, \{u,v\}\in E\}.$$
That is, there is an edge between $u_n$ and $v_m$ in $\Delta$ if and only if there is and edge between $u$ and $v$ in $\Gamma$. 
Moreover, the label $m^{\Delta}_{u_n,v_m}$ of $\{u_n, v_m\}$ is the same as the label $m_{u,v}$ of $\{u,v\}$.

The labelling $m^\Delta$ of $E_\Delta$ is, by definition, even.
It is also of $FC$-type. Indeed, we need to verify that any three vertices of $\Delta$ spanning a complete graph satisfy that at most one of the labels of the edges is greater than $2$. 
As there are no edges in $\Delta$ among vertices of the same type, if $u_n,v_m,w_l\in V_\Delta$ span a complete graph,  we must have that $u,v,w$ are three different vertices of $\Gamma$ and $n,m,l\in \bZ$. 
As $m^{\Delta}_{u_n,v_m}= m_{u,v}$, $m^{\Delta}_{v_m, w_l} = m_{v,w}$ and  $m^{\Delta}_{w_l,u_n}= m_{w,u}$ and $\Gamma$ is even FC-type, we get at most one of the $m^{\Delta}_ {u_n,v_m}, m^{\Delta}_{v_m,w_l}, m^{\Delta}_{w_l, u_n}$ is greater than 2.

\begin{lem}With the previous notation, 
$G_\Delta \cong \ker \rho$ via $v_n\mapsto x^nvx^{-n}$.
\end{lem}
\begin{proof}
We use the Reidemeister-Schreier procedure (see \cite{LyndonSchupp}) to obtain a presentation of~$K$. 
Write $V = B\sqcup L \sqcup \{x\}$. 
So the retraction map is given by:
$$\rho \colon G_{\Gamma} \to \bZ, \quad \forall a\in B\cup L: a\mapsto 0, \; x\mapsto 1,$$
with $K = ker(\rho)$.

The set $T = \{x^{i} \mid i \in \bZ\}$ gives a Schreier transversal for $K$ in $G_\Gamma$.
The set of generators for $K$ is
$Y = \{tv(\overline{tv})^{-1}\mid t\in T,\; v\in V,\; tv \not\in T \}$ where $\overline{w}$ is the representative of $w$ in $T$. 
Let us compute the set $Y$. 
For $v=x$, $x^{i}x (\overline{x^{i}x})^{-1} =1$. 
For  $v=a$ with  $a\in B\cup L$, let $a_i \coloneqq  x^{i}a(\overline{x^{i}a})^{-1} = x^{i}ax^{-i}$.
Therefore we get that the set
$$Y = \{a_i \coloneqq x^{i}ax^{-i}\mid a\in L \cup B, \; i \in \bZ\},$$
gives a set of generators for $K$.

Denote by $R$ the set of relations of the defining presentation of $G_\Gamma$. To obtain relations for~$K$, rewrite each $trt^{-1}$ for $t\in T$ and $r \in R$ using generators in~$Y$.

Write any $t\in T$ as $x^i$ for some $i\in \bZ$. 
We divide the relations in $R$ in two types:
\begin{itemize}
\item[(i)] relations involving only elements of $L\cup B$: i.e. of the form $r = (ab)^m (ba)^{-m}$ where $a,b\in L\cup B$,
\item[(ii)] relations involving $x$: i.e. of the form  $r = (ax)^{k_a}(xa)^{-k_a}$
with $a\in L$. 
\end{itemize}
In case (i) we have $trt^{-1} = x^i((ab)^m (ba)^{-m})x^{-i}$. 
Introducing $x^ix^{-i}$ between letters, and recalling that $a_i = x^{i}ax^{-i}$ we obtain:
$$trt^{-1} = (a_i b_i)^m (b_i a_i)^{-m}
$$which is an even Artin relation, for the pair $a_i,b_i$ for all $i\in \bZ$, with the same label as the Artin relation for the pair $a,b$.
    
In case (ii) we have $trt^{-1} = x^i((ax)^{k_a}(xa)^{-k_a})x^{-i}$. 
Again we put $x^ix^{-i}$ between letters, and use $a_i = x^{i}ax^{-i}$ to obtain:
$$trt^{-1} = a_i a_{i+1} \dots a_{i+k_a-1}(a_{i+1} a_{i+2} \dots a_{i+k_a})^{-1}.
$$ 
The presentation for $K$ is given as:
$$K = \langle Y \mid S \rangle,$$
where $Y = \{a_{i} = x^{i}ax^{-i}\mid a\in L\cup B, \; i \in \bZ\}$, and the relations are described as below:
\begin{itemize}
\item[(i)] if $a,b\in L\cup B$ satisfy $(ab)^m = (ba)^{m}$, then for all $i\in \bZ$: $(a_{i}b_{i})^m  = (b_{i}a_{i})^{m}$
\item[(ii)] if $a\in L$ and $x$ satisfy $(ax)^{k_a} = (xa)^{k_a}$ then for all $i\in \bZ$:
$a_{i}a_{i+1}\dots a_{i+k_a-1} = a_{i+1} a_{i+2} \dots a_{i+k_a}$.
\end{itemize}

We can use the type (ii) relations to simplify our presentation. 
If $a\in L$ and $x$ satisfy $(ax)^{k_a} = (xa)^{k_a}$, then any $a_i$ is a product of $a_{0}, a_{1}, \dots,  a_{k_a-1}.$ 
Indeed, setting $\sigma_{a} = a_{0} a_{1} \cdots  a_{k_a-1},$
we obtain:
\begin{equation}\label{eq: conjugates of a by x}
a_{l} = \sigma_a^{-q} a_{r} \sigma_a^{q},
\end{equation} 
where $l = k_a\cdot q + r$ with $0 \leq r < n$. 
Observe that if $k_a=1$ then $\sigma_a= a_0$ and $a_l=a_0$ for all $l$.

We can use Tietze transformations to eliminate all generators $a_i$, $i\not\in \{0, 1,\dots, k_a-1\}$ and the relations of type (ii).
We obtain a new presentation with generating set
$$V_\Delta = \{a_{j} \mid a\in \lk(x), \; 0 \leq j \leq k_a - 1 \text{ in } \bZ\} \cup \{b_j \mid b\in B, j\in \bZ\}.$$ 
To future use, we set the following terminology.

We need to examine what happens with  relations in case (i). 
Let us examine what is the effect of the previous Tietze transformations on  $r= (a_jb_j)^k=(b_ja_j)^k$.
 We have several cases. Note that if $B\neq \emptyset$, then we are under hypothesis (b):
\begin{enumerate}
\item[(i)] $a,b\in B$.
In this case, $r$ is unaltered under the Tietze transformations as none of the generators involved is eliminated.
\item[(ii)] $a\in L$, $b\in B$.
This case only can happen if we are in case (b) and thus $k_a=1$ and we have that $a_i=a_0$ for all $i\in \bZ$. Thus $r$ becomes $(a_0b_j)^k = (b_ja_0)^k$.
\item[(iii)] $a,b\in L$. We have here several subcases.
\begin{itemize}
\item under hypothesis (b):  we have that $k_a=k_b=1$ and then $r$ becomes $(a_0b_0)^k = (b_0a_0)^k$.
\item under hypothesis (a): if $k>1$, then because of the FC-condition , $k_a=k_b=1$ and then $r$ becomes $(a_0b_0)^k = (b_0a_0)^k$.
\item under hypothesis (a): if $k=1$, then because of the FC-condition, at least one of $k_a$ and $k_b$ is equal to 1. 
If both are equal to 1, then $R$ becomes $a_0b_0 = b_0a_0$.
If, say $k_a>1$, then $r$ becomes $\sigma_a^q a_s \sigma_a^{-q} b_0 = b_0 \sigma_a^q a_s \sigma_a^{-q}$  where $j = k_a\cdot q + s$ with $0 \leq s < k_a$. 
Note that for $j\notin \{0,1,\dots, k_a-1\}$, $r$ is a consequence of $a_0b_0=b_0a_0, \dots, a_{k_a-1}b_0 = b_0a_{k_a-1}$ and thus those relations can be eliminated.
\end{itemize}
\end{enumerate}

It is straightforward to check that the presentation that we obtain is the presentation of the Artin group $G_\Delta$ with $\Delta$ given above.
\end{proof}

Assume that $B = \emptyset$. Since the relations in $K$ come from the relations between elements of $L$, we obtain immediately the following corollary.

\begin{cor}
If $G_L$ is free and $B = \emptyset$, then the kernel $K$ is free as well, on $\sum_{a \in A} k_a$ generators, where $2k_a$ is the label of the edge in $\Gamma$ for the pair $x,a$ with $a\in L$.
\end{cor}

The following is an important observation.
\begin{lem}
If $P$ is index-parabolic in $G_\Delta$, then $P$ is parabolic in $G_\Gamma$.
\end{lem}

\subsection{Kernel of a retraction onto the complement of a vertex}
\label{subsec: kernel co-vertex}
Let $z\in V$ and $\rho\coloneqq \rho_{V-\{z\}}\colon G_{\Gamma} \to G_{V-\{z\}}$ the associated retraction.
In \cite{BMP} it is shown that  $K=\ker \rho$ is a free group and they give an explicit description of the basis. 
We shall now recall the construction in the specific case where $m_{u,z}=2$ for all $u\in L=\lk(z)$, as this simplifies our description.

Following \cite{BMP}, a set $\cN_L$ of normal forms for elements in $G_L$ is described.
There first a subset $L_1=\{u\in L: m_{u,z}=2\}$ of $L$ is defined. 
Note that in our situation $L=L_1$.
Then a normal form $\cN_1$ for $G_{L_1}$ is fixed.
The set $\cN_L$ is  defined in this case as $\cN_1$ and following the notation of \cite[Paragraph before Lemma 3.6]{BMP} in this case, one has that $T_0^*$ is the empty set, $T_0=\{1\}$, and $T= \ker \rho_L$ where $\rho_L \colon G_{V-\{z\}}\to G_L$ is the canonical retraction. 

Now $\{z\}\times T$ is a free basis of $\ker{\rho}$ (See \cite[Proposition 3.16]{BMP}) and we can identify $z_t\coloneqq (z,t)$ with $tzt^{-1}$.

\section{Intersection of parabolics}
\label{sec: intersection}
The next lemma essentially proves that the intersection of parabolic subgroups on Artin groups based on graphs with two vertices is parabolic.
It exemplifies some of the ideas used in the theorem of this section.
\begin{lem}\label{lem: intersections on two gen}
Let $\Gamma=(V=\{a,x\}, E = \{a,x\}, m)$ be an Artin graph with $m_{a,x}=2k$ for some $k\geq 1$.
Let $g\in G_\Gamma$. 
Then $\langle a \rangle \cap g \langle a \rangle g^{-1}$ is either equal to $\langle a \rangle$ or trivial.
\end{lem}
\begin{proof}
Let $\rho_x\colon G_\Gamma \to \langle x \rangle$.
Both $\langle a \rangle$ and $g\langle a \rangle g^{-1}$ lie on $\ker \rho_x$.
From subsection \ref{subsec: kernel vertex} we know that $\ker \rho_x$ is free with basis $a_0, a_1,\dots, a_{k-1}$ where $a_i=x^iax^{-i}$.
Write $g = h x^s$ where $s=\rho_x(g)$ and $h\in \ker \rho_x$.
Following Equation \ref{eq: conjugates of a by x}, we have that $x^sax^{-s} = \sigma_a^{l(s)}a_r \sigma_a^{-l(s)}$ for some $l(s)\in \bZ$, $0\leq r <k$ and $\sigma_a= a_0a_1a_2\cdots a_{k-1}$.
In particular, $$\langle a \rangle \cap g\langle a \rangle g^{-1}= \langle a_0 \rangle \cap h \sigma_a^{l(s)}\langle a_r \rangle \sigma_a^{-l(s)}h^{-1}.$$
Now, the intersection is trivial if $r\neq 0$. 
If $r=0$, as $\langle a_0\rangle$ is a malnormal subgroup (even more a free factor) of $\ker \rho_x$, 
the intersection is trivial unless $h\sigma_a^{l(s)}\in \langle a_0\rangle$,
 and in that case the intersection is the whole $\langle a_0\rangle = \langle a \rangle$. 
\end{proof}

The following theorem says that the class of even FC-type Artin graphs satisfies the condition of Equation \eqref{eq: base of induction}.
\begin{thm}\label{thm: FC implies class C}
Let $\Gamma=(V,E,m)$ be an even FC-type finite Artin graph.
Let $A\subseteq V$, such that for all $x\in V- A$, $V = \st (x)$.
Let $g\in G$.
Then either $G_A=gG_Ag^{-1}$ or there is $B\subsetneq A$ such that $G_A\cap gG_Ag^{-1}\leqslant G_B$.
\end{thm}
\begin{proof}
If $A$ is empty, then $G_A=\{1\}$ and $G_A=gG_Ag^{-1}$. So we assume that $A$ is non-empty.

Let $N$ be the number of edges from $A$ to $V- A$ with label $>2$. We will argue by induction on $N$.

If $N=0$, then for all $x\in V- A$ and all $a\in A$, the label of $\{x,a\}$ is $2$, and hence $G = G_{V- A}\times G_A$ and $gG_Ag^{-1}=G_A$ for all $g\in G$.

So assume that $N>0$ and the theorem holds for smaller values of $N$.

Consider first the case where, there is $x\in V-A$ and $a\in A$ such that the label of $\{x,a\}$ is $2k_a$ with $k_a >1$ and  $A\subseteq \st(a)$. 
In this case $\st(x) = \st(a) = V$, and for all $z\in Z \coloneqq  V-\{a,x\}$
we have that $m_{x,z}=m_{a,z}=2$, which yields $G_\Gamma= G_{\{x,a\}}\times G_Z$.
Write $g=(g_1,g_2)$ with $g_1\in G_{\{x,a\}}$ and $g_2\in G_Z$.
Then $G_A\cap gG_Ag^{-1}$ is 
equal to  
the direct product of the subgroup $\langle a \rangle \cap g_1\langle a\rangle g_1^{-1}$ of the direct factor $G_{\{x,a\}}$ and the subgroup $G_{A-\{a\}}\cap g_2 G_{A-\{a\}}g_2^{-1}$ of the direct factor $G_Z$.
By Lemma \ref{lem: intersections on two gen} we have that $\langle a \rangle \cap g_1\langle a\rangle g_1^{-1}$ is either trivial or equal to $\langle a \rangle$.
Let $A'= A-\{a\}$.
Note that for all $z\in Z-A'$, $\st(z)=Z$ and the number of vertices $z\in Z-A'$ with an edge with label $>2$ is less than $N$ (as $Z$ spans a subgraph of $\Gamma$ that consists of deleting the vertices $x,a$ and the edge $\{x,a\}$).
Therefore by induction, either $G_{A'}=g_2 G_{A'}g_2^{-1}$ or $G_{A'}\cap g_2 G_{A'}g_2^{-1}$ is contained on a parabolic subgroup over a proper subset of $A'$.
The theorem  follows in this case.

So let us consider the case where  there is $x\in  V- A$ and $a\in A$ such that the label of $\{x, a\}$ is $2k_a$ with $k_a >1$ (in particular $N\geq 1$), $ A\not\subseteq \st(a)$ and that the theorem holds for smaller values of $N$.
We remark that the condition $A\not\subseteq \st(a)$ will not be used until Case 3.2 below.

Recall from the previous section, that there exists a finite Artin graph $\Delta$, such that $\ker \rho_{x}$ is isomorphic to $G_\Delta$.
We follow the notation of Subsection \ref{subsec: kernel vertex}.
We have that $V_\Delta = \{z_0,\dots, z_{k_z}-1 : z\in V-\{x\}\}$, $E_\Delta = \{\{u_i, v_j\}\subseteq V_\Delta: u_i\neq v_j \text{ and } \{u,v\}\in E\},$ and $m^{\Delta}_{u_i,v_j}=m_{u,v}$.
Recall that the vertices of $V_\Delta$ are indexed.
We will write $A_0$ to denote the vertices of type $v\in A$ and index $0$, i.e.  $A_0=\{b_0: b\in A\}$.
Observe that the vertices $y_i$ of $V_\Delta-A_0$ such that $\lk(y)$ does not contain $A_0$ are exactly the vertices 
$b_1,\dots, b_{k_b-1}$ with $b\in A$ and $k_b>1$. 
Indeed, if $k_b>1$, $b_0,\dots, b_{k_b-1}$ span a subgraph with no edges of $\Delta$ and thus $b_0\notin\lk(b_i)$ for $i>0$. On the other hand, if $y\in V-A$, as $\st(y)=V$, we have that $k_y=1$ (since $y,x,a$ form a triangle) and then in $\Delta$ we only have a vertex $y_0$ of type $y$ and by definition of $\Delta$, we  have that $\st(y_0)=V_\Delta$.

 We note that $G_{A_0}$ is an index-parabolic subgroup of $G_\Delta$ and it is equal to the  subgroup $G_A$ of $G_\Gamma$.

Write $g= hx^s$ where $h\in \ker \rho_{x}$ and $s=\rho_{x}(g)$.
Let $Q=x^s G_A x^{-s}\leqslant \ker \rho_x$. 
We note that $Q$ might not be a parabolic subgroup of $\ker \rho_x$ although we can give a very precise description using Equation \eqref{eq: conjugates of a by x}:
 $Q$ is generated by $\{ x^s b x^{-s} :b\in A \}$. If
$k_b=1$, then $x^s b x^{-s} = b_0$. 
If $k_b>1$ then  $x^s b x^{-s}$ is equal to $\sigma_b^{l(s,b)} b_i \sigma_b^{-l(s,b)}$ where $i\in \{0,\dots, k_b-1\}$, $i\equiv s \mod k_b$, $l(s,b)\in \bZ$ and $\sigma_b=b_0b_1 \dots b_{k_b-1}$.

Now $G_A\cap gG_Ag^{-1} = G_{A_0} \cap hQh^{-1}$. Note that even if $Q$ is not parabolic, 
we are reduced to show that either $G_{A_0} = hQh^{-1}$ or that $G_{A_0} \cap hQh^{-1}$ is contained on a $\Delta$-parabolic subgroup over a proper subset of $A_0$. Indeed, in the latter case, as $\Delta$-parabolics over subsets of $A_0$ are $\Gamma$-parabolics, we also get that $G_A\cap gG_Ag^{-1}$ is contained on a $\Gamma$-parabolic subgroup over a proper subset of $A$.

We consider three cases:

{\bf Case 1: $s=0$}.  Then $Q=G_{A_0}$ is a parabolic subgroup of $G_\Delta$. By Lemma \ref{lem: reducing subgraphs} we can reduce the problem to a subgraph $\Delta'$ of $\Delta$ and $h'\in G_{\Delta'}$  such that $G_{A_0}\cup h'G_{A_0}(h')^{-1}$ is not contained in a proper parabolic subgroup of $G_{\Delta'}$. We need to show that either $G_{A_0}=h'G_{A_0}(h')^{-1}$ or $G_{A_0}\cap h'G_{A_0}(h')^{-1}$ is contained in $\Delta'$ parabolic subgroup over a proper subset of $A_0$. 

If  $b_{i}\in V_{\Delta'}$ for some $b\in A$, $k_b>1$ and $i>0$, then Lemma \ref{lem: link exterior} implies that  $G_{A_0}\cap h'G_{A_0}(h')^{-1}$ is contained in a $\Delta'$-parabolic subgroup over a proper subset of $A_0$ and we are done.
So, we can assume that $V_{\Delta'}\subseteq A_0 \cup \{y_0: y\in V-A\}$.
Note that in this case, $\Delta'$ is a finite Artin graph of even FC-type, 
for all $y_0\in V_{\Delta'}-A_0$ we have that $\st(y_0)=V_{\Delta'}$ and
the number of edges from $V\Delta'- A_0$ to $A_0$ with label $>2$ is less than $N$ (in fact, it is less than $N$ minus the number of edges $\{x,b\}$ with label $> 2$).
By induction, we get that either $ G_{A_0} = {h'}Q(h')^{-1}$ or  $G_{A_0} \cap {h'}Q(h')^{-1}$ is contained in a $\Delta'$-parabolic subgroup over a proper subset of $A_0$ and we are done.

{\bf Case 2:  $s\notin k_a\bZ$.} Then $\rho_{A_0}(hQh^{-1})\leqslant G_{A_0- \{a_0\}}$ and therefore $G_{A_0}\cap hG_Qh^{-1} \leqslant G_{A_0- \{a_0\}}$ and the lemma holds.

{\bf Case 3:  $s\in k_a\bZ$, $s\neq 0$.}
Recall that $Q$ is generated by $\{ x^s b x^{-s} :b\in A \}$, and the element $x^s bx^{-s}$ is equal (in $G_\Delta$) to
$\sigma_b^{l(s,b)} b_{i(s,b)} \sigma_b^{-l(s,b)}$ where $ i(s,b) \equiv s \mod  k_b$,  $l(s,b)\in \bZ$  and $\sigma_b = b_0b_1\dots b_{k_b-1}$(see  Equation\eqref{eq: conjugates of a by x}).
If some $i(s,b)\neq 0$, we lie in Case 2. So we assume that $i(s,b)=0$ for all $b\in A$.
 
For simplifying our notation and arguments\footnote{Below a standard parabolic subgroup $G_D$ is defined, and an advantage of using $\phi$ is that $\sigma_a\notin G_D$ but $a_1$ is in $ G_D$ .}, consider the automorphism 
$$\phi\colon G_{\Delta}\to G_{\Delta}\qquad  \phi(v)=\begin{cases}v & v\neq a_1\\
\sigma_a & v=a_1.\end{cases} $$ 
We need to show that this is well defined.
By construction, the only edges of $\Delta$ adjacent to $a_1$ are of the form $\{z_0,a_1\}$ with $z\in \lk_\Gamma(a)-\{x\}$.
Moreover, as $\st(x)=V$ and $k_a>1$, necessarily  $m_{z_0,a_1}=2$ for $z\in \lk_\Gamma(a)-\{x\}$.
Thus, we need to check that $\sigma_a$ commutes with $z_0$, $z\in \lk_\Gamma(a)-\{x\}$.
But this holds, as by construction $m_{z_0,a_i}=2$ for all $i=0,1,\dots, k_a-1$ (recall that $\sigma_a=a_0a_1\cdot a_{k_a-1}$).
Thus $\phi$ is well-defined. It is easy to check that $\phi$ is bijective.

We apply $\phi^{-1}$ to $G_{A_0}$, $Q$ and $h$, and we get $G_{A_0}$, 
\begin{equation}\label{eq: generators of x^sGA_0x^-s}
P=\langle 
\{a_1^{l(s,a)}a_0 a_1^{-l(s,a)}\}
\cup 
\{\sigma_b^{l(s,b)}  b_0 \sigma_b^{-l(s,b)}: b_0\in A_0-\{a_0\}\}\rangle
\end{equation}
and $f=\phi^{-1}(h)$ respectively.
Note that as $G_{A_0}$ is fixed by $\phi$, 
we have that $G_{A_0}=hQh^{-1}$ if and only  if $G_{A_0}=fP f^{-1}$ and that $G_{A_0}\cap hQh^{-1}$ is contained in a $\Delta$-parabolic subgroup over a proper subset of $A_0$ if and only if the same holds for $G_{A_0}\cap fPf^{-1}$.
For simplicity, we set $l=l(s,a)$.
Note that as $s\neq 0$, $l\neq 0$.

Let $D= V_\Delta-\{a_0\}$ and $\rho_D$ the corresponding retraction.
Then $G_\Delta = \ker \rho_D \rtimes G_D$.
Let $\rho_{a_1}\colon G_\Delta\to \langle a_1\rangle$ be the canonical retraction.
We have now two subcases.

{\bf Case 3.1: $\rho_{a_1}(fa_1^l) \neq 0$}.

We are going to show that 
$$(G_{A_0}\cap \ker \rho_D)\cap (fPf^{-1}\cap \ker \rho_D)=\{1\}.$$
This implies that $G_{A_0}\cap fPf^{-1}\leqslant G_D$ and therefore, $G_{A_0}\cap fPf^{-1}\leqslant G_{A_0\cap D}$ and we are done, as $A_0\cap D = A_0-\{a_0\}$.
Note that 
$$(G_{A_0}\cap \ker \rho_D) =
\langle a_0^{G_{A_0}}\rangle 
\quad\text{  and }\quad
(fPf^{-1}\cap \ker \rho_D)= \langle (fa_1^l a_0 a_1^{-l}f^{-1})^{fPf^{-1}}\rangle.$$

Let $L=\lk_\Delta(a_0)$. As $k_a>1$ every vertex in $L$ commutes with $a_0$. 
Let $T=\ker (\rho_L \colon G_{D}\to G_L)$.
Here $\rho_L$ is the canonical retraction.
Recall from subsection \ref{subsec: kernel co-vertex} that $\ker \rho_D$ is free with free basis $\{ta_0 t^{-1} : t\in T\}$.
Note that if $g\in G_D$, then there are unique  $g_L=\rho_L(g)$ and $g'\in \ker \rho_L$ such that $g=g'g_L$ and $ga_0g^{-1}= g'a_0 g'^{-1}$.

{\bf Claim 1:}  $\langle a_0^{G_{A_0}}\rangle$ is free with basis $T_{A_0}=\{ta_0t^{-1} : t\in \ker \rho_L \cap G_{A_0}\}$.

Notice that $\langle T_{A_0}\rangle \leqslant \langle a_0^{G_{A_0}}\rangle$. 
So it is enough to show that $\langle a_0^{G_{A_0}}\rangle \leqslant\langle T_{A_0}\rangle $ to prove Claim 1.
In order to show it, let $g\in G_{A_0}$ and we need to show that $ga_0g^{-1}\in  \langle T_{A_0}\rangle.$
Write $g$ as $g = g_1a_0^{m_1}g_2a_0^{m_2}\dots g_n a_0^{m_n}$ where $n\geq 0$, $g_i\in (G_D\cap G_{A_0})-\{1\}$ for $i=1,2,\dots, n$,  $m_i\in \bZ-\{0\}$ for $i=1,2,\dots, n-1$ and $m_n\in \bZ$.
We can further write $g_i$ as $c_ih_i$ where $h_i\in \ker \rho_D$ and $c_i\in G_L$, that is
$$g = c_1h_1a_0^{m_1}c_2h_2a_0^{m_2}\dots c_n h_n a_0^{m_n},$$
which rewriting $c_1 \cdots c_i h_i c_i^{-1}\cdots c_1$ as $h_i'$ and using that the $c_i$'s commute with $a_0$, we get that
$$g = h_1'a_0^{m_1}h'_2a_0^{m_2}\dots h'_n a_0^{m_n}c_1c_2\cdots c_n,$$
notice that $ga_0g^{-1}$ is equal to $g'a_0(g')^{-1}$ where 
$$g' = h_1'a_0^{m_1}h'_2a_0^{m_2}\dots h'_n.$$
We can write $g'a_0(g')^{-1}$ as a product of elements of  $T_{A_0}=\{ta_0t^{-1} : t\in \ker \rho_L \cap G_{A_0}\}$. Indeed:
\begin{align*}
g'a_0(g')^{-1} & = (h_1'a_0 (h_1')^{-1})^{m_1} \cdot (h_1' h_2' a_0 (h_1'h_2')^{-1})^{m_2} \cdots (h_1'\cdots h_{n-1}' a_0 (h'_{1}\cdots h'_{n-1})^{-1})^{m_{n-1}} \cdot \\
&\cdot h_1'\cdots h_n' a_0 (h_1'\cdots h_n')^{-1} \cdot \\
&\cdot (h_1'\cdots h_{n-1}' a_0 (h'_{1}\cdots h'_{n-1})^{-1})^{-m_{n-1}}\cdots (h_1' h_2' a_0 (h_1'h_2')^{-1})^{-m_2}\cdot(h_1'a_0 (h_1')^{-1})^{-m_1} .
\end{align*}
This completes the proof of Claim 1.

{\bf Claim 2:} $\langle (fa_1^l a_0 a_1^{-l}f^{-1})^{fPf^{-1}}\rangle$ is free with basis $T_P= \{ta_0t^{-1}: t \in f'Pa_1^l\cap \ker \rho_L\}$ where $f'$ is the unique element of $\ker \rho_L$ such that $f=f'\rho_L(f)$.

The proof of the claim is very similar to the previous one. 
One has that $\langle T_P\rangle \leqslant \langle (fa_1^l a_0 a_1^{-l}f^{-1})^{fPf^{-1}}\rangle$
 so it is enough to show that for any $g\in P$ the element 
 $$(fgf^{-1})(fa_1^la_0a_1^{-l}f^{-1})(fgf^{-1})
 =fga_1^la_0a_1^{-l}g^{-1}f^{-1}$$
 lies in $\langle T_P\rangle$.
Recall that from Equation \ref{eq: generators of x^sGA_0x^-s} that a generating set of  $P$ is 
$$\{a_1^la_0a_1^{-l}\}\cup \{\sigma_b^{l(s,b)}b_0\sigma_b^{-l(s,b)}: b\in A_0-\{a_0\}\}.$$
 In a similar way as before, we can write $g$ as 
 $$g=  c_1h_1(a_1^la_0a_1^{-l})^{m_1}c_2h_2(a_1^la_0a_1^{-l})^{m_2}\dots c_n h_n (a_1^la_0a_1^{-1})^{m_n}$$
  where $n\geq 0$, $c_i\in P\cap G_L$ and  $h_i\in P\cap \ker\rho_L$ for $i=1,\dots n$. 
Let $f_L=\rho_L(f)$.
Rewriting $f_Lc_1 \cdots c_i h_i c_i^{-1}\cdots c_1^{-1} f_L^{-1}$ as $h_i'$ and using that the $c_i$'s and $f_L$ commute with $a_0, a_1$, we get that 
 $$fg =f' h_1'(a_1^la_0a_1^{-l})^{m_1}h'_2(a_1^la_0a_1^{-l})^{m_2}\dots h'_n (a_1^la_0a_1^{-l})^{m_n}f_Lc_1c_2\cdots c_n.$$
Now notice that
$$fga_1^{l}a_0 a_1^{-l}g^{-1}f^{-1} = f'g'a_1^{l}a_0 (f'g'a_1^{l})^{-1}$$
where 
$$ g'=  h_1'(a_1^la_0a_1^{-l})^{m_1}h'_2(a_1^la_0a_1^{-l})^{m_2}\dots h'_n.$$

Now we can write $f'g'a_1^{l}a_0 (f'g'a_1^{l})^{-1}$ as a product of elements of  $T_{P}=\{ta_0t^{-1} : t\in \ker \rho_L \cap (f'Pa_1l)\}$.
\begin{align*}
f'g'a_1^{l}a_0 (f'g'a_1^{l})^{-1} & = (f'h_1'a_1^la_0 (f'h_1'a_1^l)^{-1})^{m_1} \cdot ((f'h_1' h_2'a_1^l) a_0 (f'h_1'h_2'a_1^l)^{-1})^{m_2} \cdot \\&\cdots ((f'h_1'\cdots h_{n-1}'a_1^l) a_0 (f'h'_{1}\cdots h'_{n-1}a_1')^{-1})^{m_{n-1}} \cdot \\
&\cdot (f'h_1'\cdots h_n'a_1^l) a_0 (f'h_1'\cdots h_n'a_1^l)^{-1} \cdot \\
&\cdot ((f'h_1'\cdots h_{n-1}'a_1^l) a_0 (f'h'_{1}\cdots h'_{n-1}a_1^l)^{-1})^{-m_{n-1}}\cdot \\ 
&\cdots ((f'h_1' h_2'a_1^l) a_0 (f'h_1'h_2'a_1^l)^{-1})^{-m_2}\cdot((f'h_1'a_1^l)a_0 (f'h_1'a_1^l)^{-1})^{-m_1} .
\end{align*}
This completes the proof of Claim 2.

Now, if $\rho_{a_1}(fa_1^l)\neq 0$, then $T_{A_0}\cap T_P=\emptyset$ and both are subsets  of a free basis of $\ker\rho_D$. Therefore $\langle T_{A_0}\rangle \cap \langle T_P\rangle = \{1\}$.

{\bf Case 3.2: $\rho_{a_1}(fa_1^l) = 0$.}
Note that $G_{A_0}\leqslant \ker \rho_{a_1}$ and $fPf^{-1}\leqslant \ker \rho_{a_1}$.
 As every  $z\in \lk(a_1)$ commutes with $a_1$, we are in case (b) of subsection \ref{subsec: kernel vertex} and $\ker \rho_{a_1}$ is isomorphic to $G_\Lambda$ where $\Lambda$ is an even, FC-type, Artin graph (possibly infinite).
Recall that 
$$V_\Lambda = \{b_{i,0}: b_i\in \lk_\Delta(a_1)\} \cup \{z_{i,j}: z_i\in V_\Delta-\st_\Delta(a_1), j\in \bZ\}$$ and there is an edge $\{v_{i,j}, u_{s,t}\}$ in $\Lambda$ if and only if there is and edge $\{v_i, u_s\}$ in $\Delta$ and the label of both edges is the same.

Let $A_{0,0}=\{b_{0,0}: b_0\in A_0\}$ the vertices of $\Lambda$ of level $0$ and type $A_0$.
 Note that $G_{A_{0,0}}\leqslant G_\Lambda$ is the subgroup $G_{A_0}$ of $G_\Delta$ and the subgroup $G_{A}$ of $G_\Gamma$.

As $\rho_{a_1}(fa_1^l) = 0$, we have that $\rho_{a_1}(f)\neq 0$ (recall that $l\neq 0$).
Write $f$ as $f'a_1^\alpha$, with $\alpha=\rho_{a_1}(f)\in \bZ$.
Consider the canonical retraction $\rho_{A_{0,0}}\colon G_\Lambda\to G_{A_{0,0}}$.
Now, we have that $fPf^{-1}$ is equal to $f' P' (f')^{-1}$
where $P'$ is generated by 
$$\{a_{0,0} \}\cup \{\tau_{c}^{l(s,c)} a_1^\alpha c_{0} a_1^{-\alpha}\tau_{c}^{-l(s,c)} : c\in A-\{a\} \} $$
where $\tau_{c}= a_1^{\alpha}\sigma_{c}a_1^{-\alpha}$ 
is some element of $\langle \{ v_{i,j} : v_{i,j}\text{ of type }c_i\in V_\Delta\} \rangle$.
Moreover, using Equation \ref{eq: conjugates of a by x} in the setting of $\rho_{a_1}$ we have that
 $$a_1^\alpha c_{0} a_1^{-\alpha}= \beta_c^{l(\alpha,c)} c_{0,i(\alpha,c)} \beta_c^{-l(\alpha,c)}$$
 for some word $\beta_c\in \langle \{ v\in V_{\Lambda}: \, $v$\text{ of type  } c_{0}\}\rangle$ and some $i(\alpha,c)\in \bZ$.
Observe that 
$$
\rho_{A_{0,0}}
(\tau_c^{l(s,b)}a_1^{\alpha} c_{0}a_1^{-\alpha}\tau_c^{-l(s,b)})
= 
\begin{cases}
c_{0,0} & \text{if }i_{\alpha,c}=0\\
1 & \text{otherwise}.
\end{cases}
$$

Recall that we are assuming $A\not\subseteq \lk(a)$ and therefore there exists some $b\in A$ such that $b$ is not linked to $a$. 
Thus $b_0$ is not linked to $a_i$, $i=0,1,\dots, k_a-1$ in $\Delta$.
As $b_0$ is not linked to $a_1$, we have that $a_1^\alpha b_{0} a_1^{-\alpha}=b_{0,\alpha}$.
And we get that $\rho_{A_{0,0}}(P') \leqslant G_{A_{0,0}-\{b_{0,0}\}}.$
In particular $G_{A_{0,0}}\cap  f'P'(f')^{-1}\leqslant \rho_{A_{0,0}}(f') G_{A_{0,0}-\{b_{0,0}\}}\rho_{A_{0,0}}(f')^{-1}$.
Note that $G_{A_{0,0}-\{b_{0,0}\}}=G_{A_0-\{b_0\}}$.
So, there  is $d\in G_{\Delta}$ such that 
$\rho_{A_{0,0}}(f') G_{A_{0,0}-\{b_{0,0}\}}\rho_{A_{0,0}}(f')^{-1}=d G_{A_0-\{b_0\}}d^{-1}$,
and thus $G_{A_0}\cap fPf^{-1}$ is contained in $dG_{A_0-\{b_0\}}d^{-1}$, a parabolic over a proper subset of $A_0$.
This completes the proof in this case.
\end{proof}
\begin{proof}[Proof of Theorem \ref{thm: intersections}]
Let $\cC$ be the class of finite, even, FC-type Artin graphs. 
Then $\cC$ is closed under subgraphs and satisfies \eqref{eq: base of induction} by Theorem \ref{thm: FC implies class C}.
The theorem now follows from Proposition \ref{prop: reduction to stars}.
\end{proof}

\begin{cor}\label{cor: arbitrary intersection}
Let $\Gamma = (V,E,m)$ be an even, finite Artin graph of FC-type Then any arbitrary intersection of parabolic subgroups in $G_\Gamma$ is a parabolic subgroup.
\end{cor}

\begin{proof}
Let $\cP$ be the set of parabolic subgroups in $G$.
Note that as $\Gamma$ is finite, $\cP$ is countable.
For an arbitrary indexing set $I$, we want to show that:
$$Q = \bigcap_{i\in I, P_i \in \cP} P_i$$
is a parabolic subgroup.
If $I$ is finite, the claim follows from Theorem \ref{thm: intersections} and induction. 
So, we can assume that the indexing set $I$ is countable, and we can index its elements by natural numbers. 
Write:
$$\bigcap_{i\in I} P_i = \bigcap_{n \in \bN} \left( \bigcap_{i \leq n} P_i \right),$$
and set $Q_n = \bigcap_{i \leq n} P_i$. 
We know that $Q_n$ is a parabolic subgroup for any $n$. 
Moreover we have a chain of parabolic subgroups:
$$Q_1 \supseteq Q_2  \supseteq Q_3  \supseteq \cdots $$
where the intersection of all members $Q_i$ of the chain above is equal to $Q$. 
We cannot have an infinite chain of nested distinct parabolic subgroups. Indeed, using Lemma \ref{lem: proper_inclusions_parabolics} we have that $gG_Ag^{-1} \subsetneq hG_Ah^{-1}$ implies $A \subsetneq B$. Hence there are at most $|V| + 1$ distinct parabolic subgroups in the chain above. 

Ultimately, $Q$ is an intersection of at most $|V| + 1$ parabolic subgroups and hence it is a parabolic subgroup.
\end{proof}

\vspace{1cm}

\noindent{\textbf{{Acknowledgments}}} 
The authors are grateful to Mar\'{i}a Cumplido for helpful conversations while working on this project.

 Yago Antol\'{i}n  acknowledges partial support from the Spanish Government through the ``Severo Ochoa Programme for Centres of Excellence in R\&{}D'' CEX2019-000904-S.

Islam Foniqi is a member of INdAM--GNSAGA, and gratefully acknowledges support from the Department of Mathematics of the University of Milano--Bicocca, and the Erasmus Traineeship grant 2020--1--IT02--KA103--078077.

\noindent\textit{\\ Yago Antol\'{i}n,\\
Fac. Matem\'{a}ticas, Universidad Complutense de Madrid and \\ 
Instituto de Ciencias Matem\'aticas, CSIC-UAM-UC3M-UCM\\
Madrid, Spain\\}
{email: yago.anpi@gmail.com}

\noindent\textit{\\ Islam Foniqi,\\
Università degli Studi di Milano - Bicocca\\ 
Milan, Italy\\}
{email: islam.foniqi@unimib.it}

\end{document}